\documentclass[leqno]{amsart}

\usepackage[english]{babel}
\usepackage{amscd} 
\usepackage{amssymb} 
\usepackage{amsfonts}
\usepackage{mathrsfs} 
\usepackage{amsmath} 
\usepackage{amsthm}
\usepackage{amsxtra} 
\usepackage{pb-diagram}
\usepackage{srcltx}
\usepackage{color}

\DeclareMathOperator{\ind}{\mathrm{ind}} 
\newcommand{\R}{\mathbb{R}}
\newcommand{\N}{\mathbb{N}}

\newcommand{\fix}{\mathop\mathrm{Fix}\nolimits}

\newcommand{\sign}{\mathop\mathrm{sign}\nolimits}

\newcommand{\cl}[1]{\overline{#1}}
\newcommand{\dist}{\mathop\mathrm{dist}\nolimits}
\newcommand{\dom}[1]{\mathcal{D}(#1)} 
\newcommand{\D}{\partial}
\newcommand{\X}{\times} 
\newcommand{\h}[1]{\widehat{#1}}
 
\newcommand{\Q}{\mathcal{Q}}
\newcommand{\giu}[1]{#1\sb\#} 
\newcommand{\su}[1]{#1\sp\#}
\renewcommand{\t}[1]{\widetilde{#1}}

\newcommand{\rkrs}{\mathbb{R}^k\times\mathbb{R}^s}
\newcommand{\ctnn}{C_T (N)}

\def\overstrike#1#2{{\setbox0\hbox{$#2$}\hbox to \wd0{\hss
      $#1$\hss}\kern-\wd0\box0}}
\newcommand{\avrg}[1]{\overstrike{#1}{/}}

\newtheorem{theorem}{Theorem}[section]

\newtheorem{lemma}[theorem]{Lemma}
\newtheorem{example}[theorem]{Example}
\newtheorem{proposition}[theorem]{Proposition}
\newtheorem{remark}[theorem]{Remark}

\numberwithin{equation}{section}

\author{Luca Bisconti} \author{Marco Spadini} \address{Dipartimento di
  Sistemi e Informatica, Universit\`a degli Studi di Firenze,Via S.\
  Marta 3, 50139 Firenze, Italy}

\date{\today}

\begin{document}

\title[On a class of delay separated variables equations] {Periodic
  perturbations with delay of separated variables differential
  equations}

\begin{abstract}
  We study the structure of the set of harmonic solutions to perturbed
  nonautonomous, $T$-periodic, separated variables ODEs on
  manifolds. The perturbing term is allowed to contain a finite delay
  and to be $T$-periodic in time.
\end{abstract}

\maketitle

\section{Introduction} 
In this paper we study $T$-periodic solutions to periodic
perturbations of separated variables ODEs on manifolds, allowing the
perturbing term to contain a finite delay. Namely, given $T > 0$, $r
\geq 0$ and a boundaryless smooth manifold $N \subseteq \R^d$, we
consider $T$-periodic solutions to equations of the form
\begin{equation} \label{DSVeq:0} \dot \zeta (t) = a(t)\Phi(\zeta(t))
  +\lambda \Xi \big(t, \zeta(t), \zeta(t-r)\big), \quad \lambda\geq 0,
\end{equation}
where $r>0$ is a finite time lag, $a\colon\R\to\R$ is a continuous
$T$-periodic function, $\Phi\colon N\to\R^d$ and $\Xi\colon \R\X N\X N
\to\R^d$ are given continuous tangent vector fields on $N$, in the
sense that $\Phi(\xi)$ belongs to the tangent space $T_\xi N$, for any
$\xi\in N$, and $\Xi$ is $T$-periodic in the first variable and
tangent to $N$ in the second variable, that is
\begin{equation*}
  \Xi(t, \xi, \eta)=\Xi(t+T, \xi, \eta) \in  T_\xi N,\qquad \forall 
  (t, \xi, \eta) \in \R \X N\X N.
\end{equation*}
We also assume that the average $\avrg{a}$ of $a$ is nonzero, i.e.,
\begin{equation} \label{DSVeq:average-nn} \avrg{a}:=
  \frac{1}{T}\int^T_0 a(t) dt\ne 0.
\end{equation}
Clearly, $a(t)$ can be written as $\avrg{a}+\alpha(t)$ where
$\alpha\colon\R\to\R$ is continuous, $T$-periodic and with zero
average. In this way Equation \eqref{DSVeq:0} can be obtained by the
introduction of a $T$-periodic perturbation with null average in the
coefficient $\avrg{a}$ in the following equation:
\[
\dot \zeta (t) = \avrg{a}\Phi\big(\zeta(t)\big) +\lambda \Xi\big(t, \zeta(t),
\zeta(t-r)\big), \qquad \lambda\geq 0.
\]

Our main objective is to provide information on the structure of the
set of $T$-periodic solutions of \eqref{DSVeq:0} (recall that $T>0$ is
given). More precisely, we will give conditions ensuring the existence
of a connected set of pairs $(\lambda, \zeta)$, $\lambda\geq 0$ and
$\zeta$ a $T$-periodic solution of \eqref{DSVeq:0}, such that either
$\lambda\neq 0$ or $\zeta$ is nonconstant, whose closure in an
appropriate topological space is not compact and meets the set of
pairs formed by constant solutions corresponding to $\lambda=0$.

To pursue our goal we use topological tools as the fixed point index
and the degree of tangent vector fields on manifolds (see, e.g.\
\cite{fupespa}). A deceptively natural approach would be to use
a time transformation as in e.g.\ \cite{spaSepVar} and then apply
directly one of the known results for periodic perturbations of
autonomous ODEs on manifolds as, for instance, \cite{FS09}. This naive
procedure does not work because the time-transformed perturbing term
would result in a form difficult to investigate. Our strategy, instead,
consists of recasting and combining the arguments of \cite{spaSepVar}
and \cite{FS09}. To get a general idea of how we proceed,
consider the particular case when $\Phi$ is $C^1$ (the general case
when $\Phi$ is only continuous boils down to it via an approximation
procedure). When $\alpha\equiv 0$ and the perturbation $\Xi$ does not
depend on the delay, as in \cite{fuspa:1997}, our condition is
obtained through a formula (see e.g.\ \cite{fupespa}) relating the
degree of the tangent vector field $\Psi$ to fixed point index of the
translation operator at time $T$, $P^{\Phi}_T$, associated to the
equation
\begin{equation}
  \label{DSVeq:noPertEq}
  \dot{\zeta} = \Phi(\zeta).
\end{equation}
When $a$ is constant but the perturbing term in equation
\eqref{DSVeq:0} is allowed to contain a delay, as in \cite{FS09}, one
needs to adapt this approach: the operator $P^{\Phi}_T$ is replaced
with its infinite-dimensional analogous $Q^\Phi_T$ and the result is
obtained through a formula that associates the degree of $\Phi$ with
the fixed point index of $Q^\Phi_T$. In the present paper, in order to
allow $a$ to be nonconstant, we revisit the construction in
\cite{FS09} and provide a relation (see Theorem \ref{thgrado} below)
between the degree of $\Phi$ with the fixed point index of the
\emph{infinite dimensional Poincar\'e type $T$-translation operator}
$Q^{a\Phi}_T$ associated to the separated variables equation
$\dot{\zeta} = a(t)\Phi(\zeta)$. Namely, $Q^{a\Phi}_T$ is the operator 
that associates to any element $\varphi\in C([-r,0],N)$ the function 
given by $\theta\mapsto \zeta\big(\varphi(0),\theta+T\big)$, with 
$\theta\in [-r, 0]$. Here $\zeta(p,\cdot)$ denotes the unique
solution of the following the Cauchy problem on $N$:
\begin{equation*}
 \dot \zeta=a(t)\Phi(\zeta),\quad \zeta(0)=p.
\end{equation*}
Some preliminary results strictly related to this topic can also be 
found in \cite{BIS:2011, FS09, spaSepVar}.

In order to illustrate our result, we consider two classes of
applications. The first, rather straightforward, to the set of
$T$-periodic solutions of a particular class of weakly coupled
differential equations on manifolds.  In the second, we consider delay
periodic perturbations to a family of semi-explicit
differential-algebraic equations (DAEs). More precisely, we study the
structure of the set of $T$-periodic solutions to the following
problem
\begin{equation} \label{DSVeq:0a} \left\{
    \begin{array}{l} 
      \dot x(t) =   a(t)f(x(t), y(t)) +\lambda h(t, x(t), y(t), x(t-r), 
      y(t-r)),\; \lambda\geq 0,\\
      g(x, y) = 0, \\
    \end{array} \right.
\end{equation} 
where $r$ is as in \eqref{DSVeq:0}, $f\colon U\to\R^k$, $h\colon\R\X
U\X U\to\R^k$ and $g\colon U\to\R^s$ are given continuous maps defined
on an open connected set $U\subseteq \rkrs\cong \R^d$ and assume that
$h$ is $T$-periodic in the variable $t$.  We also require that $g\in
C^{\infty}(U, \R\sp{s})$, with the property that the Jacobian matrix
$\D_2 g (p,q)$ of $g$, with respect to the last $s$ variables, is
invertible for any $(p,q)\in U$.  Observe that this assumption implies
that $0$ is a regular value for $g$. So, $g^{-1}(0)\subseteq U$ is a
closed $C^{\infty}$ submanifold of $\rkrs$ of dimension $k$.
Throughout the paper we will always denote the manifold $g^{-1}(0)$ by
$M$; in this contest, the points of $M$ will written as pairs $(p,q)$.
It is well-known (see e.g.\ \cite{dae}) that under these hypotheses it
is always possible to transform the above DAE into an equivalent ODE
of type \eqref{DSVeq:0} on the differentiable manifold $M$. Actually,
as a direct consequence of the Implicit Function Theorem, $M$ can be
locally represented as a graph of some map from an open subset of
$\R^k$ to $\R^s$ and, hence Equation \eqref{DSVeq:0a} can be locally
decoupled. However, globally, this might be false or not convenient
for our purpose (see, e.g.\ \cite{CaSp}).

\section{Preliminaries and basic notions} \label{DSVeq:section2}

In this section, we recall some basic facts and definitions about the
function spaces used throughout the paper.

Let $I\subseteq\R$ be an interval and let $X\subseteq\R^d$.  Given
$r\in\N\cup\{0\}$, the set of all $X$-valued $C^r$-functions defined
on $I$ is denoted by $C^r(I,X)$. When $I=\R$, we simply write $C^r(X)$
instead of $C^r(\R,X)$ and, when $r=0$ we simplify the notation
writing $C(I,X)$ in place of $C^0(I,X)$ and $C(X)$ instead of
$C^0(X)$. Let $T>0$ be given, by $C_T (\R^d)$ we mean the Banach space
of all the continuous $T$-periodic functions $\zeta\colon\R\to\R^d$
whereas $C_T (X)$ denotes the metric subspace of $C_T (\R^d)$
consisting of all those $\zeta\in C_T (\R^d)$ that take values in
$X$. It not difficult to prove that $C_T (X)$ is complete if and only
if $X$ is closed in $\R^d$.

Let $N\subseteq \R^d$ be a smooth differentiable manifold, and
consider the following diagram of closed embeddings:
\begin{equation} \label{DSVeq:graph}
  \begin{diagram} \node{[0,\infty)\X N} \arrow{e} \node{[0,\infty)\X
      C_T(N)} \\ \node{N} \arrow{e} \arrow{n} \node{C_T(N)} \arrow{n}
  \end{diagram}
\end{equation} 
we identify any space in the above diagram with its image.  In
particular, $N$ will be regarded as its image in $C_T (N)$ under the
embedding that associates to any $p\in N$ the function
$\overline{p}\in C_T (N)$ constantly equal to $p$. Furthermore, we
will regard $N$ as the slice $\{0\}\X N\subseteq [0, \infty)\X N$ and,
analogously, $C_T (N)$ as $\{0\}\X C_T(N)$.  Thus, if $\Omega$ is a
subset of $[0, \infty)\X C_T(N)$, then $\Omega\cap N$ represents the
set of points of $N$ that, regarded as constant functions, belong to
$\Omega$.  Namely, with this convention, we have that
\begin{equation} \label{DSVeq:convention-constan-func} \Omega\cap
  N=\big\{p \in N : (0,\cl{p})\in\Omega\big\}.
\end{equation}

\smallskip Let $\Theta\colon\R\X N\to\R^d$ be a time-dependent tangent
vector field and assume that the Cauchy problem
\begin{subequations}\label{SVDAE:eq-testCP}
  \begin{equation}\label{SVDAEs:eq-test}
    \dot \zeta =\Theta (t, \zeta),\quad t\in\R,
  \end{equation} 
  \begin{equation}
    \zeta(0)=p,
  \end{equation}
\end{subequations}
admits unique solution for all $p\in N$. Denote by
\begin{equation*}
  \mathcal{D}=\big\{ (\tau, p) \in \R \X N : 
  \textrm{the solution of \eqref{SVDAE:eq-testCP} 
    is continuable up to $t=\tau$} \big\}.
\end{equation*}
A well known argument based on some global continuation properties of
the flows (see, e.g. \cite{Lang-1}) shows that $\mathcal{D}$ is open
set containing $\{0\}\X M$.  Let $P^\Theta\colon\mathcal{D}\to N$ be
the map that associates to each $(t, p)\in \mathcal{D}$ the value
$\zeta (t)$ of the maximal solution $\zeta$ to
\eqref{SVDAE:eq-testCP}, i.e.\ $P^\Theta(t, p)=\zeta (t)$. Here and in
the sequel, given $\tau\in\R$, we denote by $P^\Theta_\tau=
P^\Theta(\tau,\cdot)$, the (Poincar\'e) $\tau$-translation operator
associated to Equation \eqref{SVDAEs:eq-test}.  So that, the domain of
$P^\Theta_\tau$ is an open (possibly empty) set formed by the points
$p \in N$ for which the maximal solution of \eqref{SVDAEs:eq-test},
starting from $p$ at $t = 0$, is defined up to $\tau$.

The remark below, borrowed by \cite{spaSepVar}, plays a crucial role
in what follows.

\begin{remark} \label{SVDAEs:remark-on-ivp} Let $\Phi \colon N
  \to\R^d$ as in \eqref{DSVeq:0}.  Consider the following Cauchy
  problems
  \begin{subequations}
    \begin{gather}
      \dot \zeta = \Phi(\zeta), \quad \zeta(0)=\zeta_0,
      \label{DSVeq:A}\\[-0.3em]
      \intertext{and} \dot \zeta = a(t)\Phi(\zeta), \quad
      \zeta(0)=\zeta_0, \label{DSVeq:B}
    \end{gather}
  \end{subequations}
  with $\zeta_0\in N$.  Let $J$ and $I$ be the intervals on which are
  defined the (unique) maximal solutions of \eqref{DSVeq:A} and
  \eqref{DSVeq:B} respectively.  Suppose also that $\Phi$ is $C^1$, so
  that the uniqueness of solutions for the above problems is
  guaranteed. Let $\xi\colon J \to U$ and $u\colon I \to U$ be the
  maximal solutions of \eqref{DSVeq:A} and \eqref{DSVeq:B}
  respectively, with $I$ and $J$ the relative maximal intervals of
  existence.

  Let $t>0$ be such that $\int_0^{l} a(s) ds\in I$ for all $l\in
  [0,t]$, then it follows that
  \begin{equation*}
    \xi(t) = u\left(\int_0^t a(s)ds\right),
  \end{equation*} 
  and hence $t \in J$.  Conversely, by using a standard maximality
  argument, one can prove that $t \in J$ implies $\int_0^ta(s)ds\in
  I$.  Assume that the average $\avrg{a}$ of $a$ is $1$. Define the
  map $\phi_a\colon J\to I$, $t\mapsto\phi_a(t)=\int^t_0a(s)ds$.
  Notice that, if $T\in J$, then $\phi_a(T)= T \in I$ and so
  $\big(\xi(T),\sigma(T)\big) = \big(x(T), y(T)\big)$. Observe that
  $\phi$ needs not be invertible unless, of course, $a(t)\neq 0$ for
  all $t\in\R$.
\end{remark}

This remark has important consequences in terms of the $T$-translation
(Poincar\'e) operators associated to the Cauchy problems
\eqref{DSVeq:A} and \eqref{DSVeq:B}. We collect them in the following
proposition.

\begin{proposition} \label{SVDAEs:equivalence-T-res} Let $P^{a\Phi}_T$
  and $P^{\Phi}_T$ be the $T$-translation operators associated to the
  Cauchy problems \eqref{DSVeq:A} and \eqref{DSVeq:B},
  respectively. Then $P^{a\Phi}_T (\zeta_0)$ is defined if and only if
  so is $P^\Phi_T(\zeta_0)$. In particular, if we assume in addition
  that the average of $a$ on $[0,T]$ is equal to $1$, we have
  $P^\Phi_T(\zeta_0)=P^{a\Phi}_T(\zeta_0)$, whenever $P^{a\Phi}_T$ or
  $P^{\Phi}_T$ is defined.
\end{proposition}

\begin{proof}
  Follows immediately from Remark \ref{SVDAEs:remark-on-ivp}.
\end{proof}

The following result is proved in \cite[Corollary 2.4]{spaSepVar}

\begin{proposition}\label{prop:iaphi=iphi}
  Let $\Phi \colon N \to \R^k$ be a $C^1$ tangent vector field, and
  let $a \colon\R\to\R$ be continuous and $T$-periodic with
  $\frac{1}{T} \int^T_0 a(s) ds = 1$.  Given an open subset $V$ of
  $N$, if $\ind(P^{a\Phi}_T, V)$ is well defined, then so is
  $\ind(P^\Phi_T, V)$ and
  \begin{equation} \label{DSVeq:iaphi=iphi} \ind(P^{a\Phi}_T, V) =
    \ind(P^{\Phi}_T, V) = \deg(-\Phi, V).
  \end{equation}
\end{proposition}

\section{The degree of tangent vector fields on manifolds and some
  related properties}\label{DSVeq:section-degree}

We now remind some facts about the notion of degree of tangent vector
fields on manifolds. Recall that if $\Theta \colon N\to\R^d$ is a
tangent vector field on the differentiable manifold $N\subseteq\R^d$
which is (Fr\'echet) differentiable at $p\in N$ and $\Theta(p) = 0$,
then the differential $\textrm{d}_{p}\Theta\colon T_{p}N\to\R^d$ maps
$T_{p}N$ into itself (see, e.g., \cite{milnor}), so that, the
determinant $\det\textrm{d}_{p}\Theta$ is defined.  In the case when
$p$ is a nondegenerate zero (i.e.\ $\textrm{d}_{p} \Theta\colon T_{p}
N\to \R^d$ is injective), $p$ is an isolated zero and
$\det\textrm{d}_{p}\Theta\ne 0$. Let $W$ be an open subset of $N$ in
which we assume $\Theta$ admissible for the degree, that is we suppose
the set $\Theta^{-1}(0)\cap W$ is compact. Then, it is possible to
associate to the pair $(w, W)$ an integer, $\deg (w, W)$, called the
degree (or characteristic) of the vector field $\Theta$ in $W$ (see
e.g.\ \cite{fupespa,difftop}), which, roughly speaking, counts
(algebraically) the zeros of $\Theta$ in $W$ in the sense that when
the zeros of $\Theta$ are all non-degenerate, then the set
$\Theta^{-1}(0)\cap W$ is finite and
\begin{equation} \label{DSVeq:deg} \deg(\Theta, W) = \sum_{{q}
    \in\Theta^{-1}(0)\cap W} \sign\, \det \textrm{d}_q\Theta.
\end{equation}
The concept of degree of a tangent vector field is related to the
classical one of Brouwer degree (whence its name), but the former
notion differs from the latter when dealing with manifolds.  In
particular, this notion of degree does not need the orientation of the
underlying manifolds. However, when $N=\R^d$, the degree of a vector
field $\deg(\Theta, W)$ is essentially the well known Brouwer degree
of $\Theta$ on $W$ with respect to $0$.  The degree of a tangent
vector field satisfies all the classical properties of the Brouwer
degree: \emph{Solution, Excision, Additivity, Homotopy Invariance,
  Normalization} etc. For an ampler exposition of this topic, we refer
e.g.\ to \cite{fupespa, difftop, milnor}.

The Excision property allows the introduction of the notion of
\emph{index} of an isolated zero of a tangent vector field.  Let $q\in
N$ be an isolated zero of a tangent vector field $\Theta\colon
N\to\R^d$. Obviously, $\deg(\Theta,V)$ is well defined for any open
set $V \subseteq N$ such that $V \cap \Theta^{-1}(0) = \{q\}$.
Moreover, by the Excision property, the value of $\deg(\Theta,V)$ is
constant with respect to such $V$'s. This common value of
$\deg(\Theta,V)$ is, by definition, the index of $\Theta$ at $q$, and
is denoted by $\textrm{i}(\Theta,q)$.  Using this notation, if
$(\Theta,W)$ is admissible, by the Additivity property we have that if
all the zeros in $W$ of $\Theta$ are isolated, then
\begin{equation*}
  \deg(\Theta, W) = \sum_{q\in \Theta^{-1}(0)\cap W} \textrm{i} (\Theta, q).
\end{equation*}
By formula \eqref{DSVeq:deg} we have that if $q$ is a nondegenerate
zero of $\Theta$, then $\textrm{i} (\Theta, q) = \sign \det
\textrm{d}_q\Theta$.

\subsection{Tangent vector fields on implicitly defined manifolds}

Let $\Pi:\R\times N\to\R^d$ be a continuous time-dependent tangent
vector field on the differentiable manifold $N\subseteq\R^d$, that is
$\Pi(t,\zeta)\in T_{\zeta}N$ for each $(t,\zeta)\in\R\times N$. Assume
that there is a connected open subset $U$ of $\R^d$ and a smooth map
$\ell:U\to\R^s$, $0<s<d$, with the property that $N=\ell^{-1}(0)$.
Suppose also that, up to an orthogonal transformation, one can realize a
decomposition $\R^d=\R^k\times\R^s$ such that the partial derivative,
$\partial_2 \ell(x,y)$, of $\ell$ with respect to the second
$s$-variable is invertible for each $(x,y)\in U$.  \smallskip

We will need the following fact.
\begin{remark}\label{remext0}
  Let $\Pi$ be as above. Since $\R\times N$ is a closed subset of the
  metric space $\R\times U$, the well known Tietze's Theorem (see
  e.g.\ \cite{Dug}) implies that there exists a continuous extension
  $\h\Pi:\R\times U\to\R^d$ of $\Pi$.
\end{remark}

This fact can be interpreted as the possibility of ``extending'' a
differential equation on $N$ to a neighborhood $U$ of $N$ in
$\R^d$. Consider, in fact, the following differential equation on $N$:
\begin{equation}\label{eq.phi.onM}
  \dot\eta= \Pi(t,\eta).
\end{equation}
Let us also consider the following ``extended'' equation on the
neighborhood $U$ of $N$ in $\R^d$:
\begin{equation}\label{eq.phitilde.onM}
  \dot\eta= \h\Pi(t,\eta),
\end{equation}
where $\h\Pi$ is any extension of $\Pi$ as in Remark \ref{remext0}.
Observe that the solutions of \eqref{eq.phi.onM} are also solutions of
\eqref{eq.phitilde.onM}. Conversely, the solutions of
\eqref{eq.phitilde.onM} that meet $N$ do actually lie on $N$ and thus
are solutions of \eqref{eq.phi.onM}.
%

Setting $\eta=(x,y)$, Equation \eqref{eq.phitilde.onM} can be
conveniently rewritten as follows:
\begin{equation}\label{eq.phi12.onM}
  \left\{
    \begin{array}{l}
      \dot x= \h\Pi_1(t,x,y),\\
      \dot y= \h\Pi_2(t,x,y),
    \end{array}\right.
\end{equation}
where,
according to the above decomposition of $\R^d$, for any $(t,\xi)\in\R\X
N$, $\xi=(x,y)$, we can write that
\begin{equation*}
  \Pi(t,\xi)=\Pi(t,x,y)=\big(\Pi_1(t,x,y), \Pi_2(t,x,y)\big)\in \rkrs.
\end{equation*}
As a direct consequence of our assumptions, it follows that
\begin{equation}\label{phi2}
  \Pi_2(t,x,y)=-\big(\partial_2\ell(x,y)\big)^{-1}\partial_1\ell(x,y)\Pi_1(t,x,y).
\end{equation}
Indeed, the condition $\Pi(t,\xi)\in T_{\xi}N$ is equivalent to
$\Pi(t,\xi)\in\ker\ell'(x,y)$, and one has that, for each
$(t,(x,y))\in\R\times N$,
\begin{equation*}
  0= \ell'(x,y)\Pi(t,x,y)= \partial_1\ell(x,y)\Pi_1(t,x,y)+\partial_2\ell(x,y)\Pi_2(t,x,y),
\end{equation*}
which implies \eqref{phi2}. Here $\ell '(x,y)$ denotes the Fr\'echet
derivative of $\ell$ at $(x,y)$.

\medskip Consider now Equation \eqref{eq.phi.onM} on $N$ again. The
simple result below shows that, when the above assumptions on $g$
hold, \eqref{eq.phi.onM} can be equivalently rewritten as the
following differential-algebraic equation:
\begin{equation}\label{eq.phi.DAE}
  \left\{
    \begin{array}{l}
      \dot x=  \h\Pi_1(t,x,y),\\
      \ell(x,y)=0.
    \end{array}\right.
\end{equation}
Here $\h\Pi$ is \emph{any} continuous extension of $\Pi$ as in Remark
\ref{remext0}, and by a \emph{solution} of \eqref{eq.phi.DAE} we mean
a pair of $C^1$ functions $x\colon J\to\R^k$ and $y\colon J\to\R^s$,
$J$ a nontrivial interval, with the property that $\dot
x(t)=\Pi_1(t,x(t),y(t))$ and $\ell\big(x(t),y(t)\big)=0$ for all $t\in
J$.

\begin{lemma}\label{lemmequiv}
  The equation \eqref{eq.phi.onM} is equivalent to the DAE
  \eqref{eq.phi.DAE}.
\end{lemma}
\begin{proof}
  Let $x\colon J\to\R^k$ and $y\colon J\to\R^s$ be $C^1$ maps defined
  on an interval $J$ with the property that
  $t\mapsto\xi(t)=\big(x(t),y(t)\big)$ is a solution of
  \eqref{eq.phi.onM}. Then, for all $t\in J$, $\dot
  x(t)=\Pi_1(t,x(t),y(t))$ and, since $\big(x(t),y(t)\big)\in N$, we
  have $\ell\big(x(t),y(t)\big)=0$.

  Conversely, let $t\mapsto\big(x(t),y(t)\big)$ be a solution of
  \eqref{eq.phi.DAE}. Then, differentiating
  $\ell\big(x(t),y(t)\big)=0$ at any $t\in J$, one gets
  \begin{equation*}
    \partial_1\ell\big(x(t),y(t)\big)\dot x(t)+\partial_2
    \ell\big(x(t),y(t)\big)\dot y(t)=0.
  \end{equation*}
  So that
  \begin{multline*}
    \dot y(t)=-(\partial_2\ell\big(x(t),y(t)\big))^{-1}\partial_1
    \ell\big(x(t),y(t)\big)\dot x(t)\\
    =-(\partial_2\ell\big(x(t),y(t)\big))^{-1}\partial_1
    \ell\big(x(t),y(t)\big)\h\Pi_1\big(t,x(t),y(t)\big).
  \end{multline*}
  Hence, on account of \eqref{phi2},
  \[
  \dot y(t)=\h\Pi_2\big(t,x(t),y(t)\big).
  \]
  The assertion follows recalling that the solution meets $N$.
\end{proof}

Let us consider the particular case in which the tangent vector field
$\Pi\colon\R\X N \to \R^d$ is defined as
\begin{equation*}
  \Pi(t,\xi) =a(t)\Phi(\xi),\quad (t,\xi)\in \R\X N,
\end{equation*}
where $\Phi\colon N\to\R^d=\rkrs$ is a continuous tangent vector field
on $N=\ell^{-1}(0)$, with $\ell\colon U\to\R^s$ as above, and
$a\colon\R\to\R$ is a continuous map.  As in Remark \ref{remext0},
Tietze's Theorem implies the existence of a continuous extension
$\h\Phi\colon U\to\R^d$ of $\Phi$ with $\h\Phi|_{N}\equiv\Psi$.

Define the map $F \colon U \to\rkrs$, as follows
\begin{equation}\label{DSVeq:F}
  F(p, q) := \big({\h\Phi}_1 (p, q), \ell(p, q)\big),
\end{equation} 
where $\h\Phi\colon U\to\R^d$ is any extension of $\Phi$ to $U$ and
${\h\Phi}_1$ is its first $\R^k$-component.  The following result (see
e.g.\ \cite[Th.\ 4.1]{CaSp}), allow us to reduce
the computation of the degree of the tangent vector field $\Phi$ on
$N$ to that of the Brouwer degree of the map $F$ with respect to $0$,
which is in principle handier. Namely, we have that
\begin{theorem} \label{DSVeq:teo1} Let $U \subseteq \rkrs$ be open and
  connected, let $\Psi$ be as above, and $F\colon U\to\rkrs$ be given
  by \eqref{DSVeq:F}. Then, for any $V\subseteq U$ open, if either
  $\deg(\Phi, N\cap V)$ or $\deg(F, V)$ is well defined, so is the
  other, and
  \begin{equation*}
    |\deg (\Phi, N\cap V)| = |\deg (F, V)|.
  \end{equation*}
\end{theorem}

\section{Poincar\'e-type translation
  operator}\label{SVADEs:section-poincare-trasl-op}

Let $N\subseteq \R^d$ be a manifold, and let $\Phi\colon N\to\R^d$ be
a tangent vector field on $N$.  Let $\Xi\colon\R\X N\X N \to \R^d$ be
continuous and tangent to $N$ in the second variable. Given $T > 0$,
assume also that $\Xi$ is $T$-periodic in $t$.  Consider the following
delay differential equation
\begin{equation}\label{dde-ro}
  \dot \zeta (t) = a(t)\Phi\big(\zeta(t)\big) +\lambda\Xi\big(t,\zeta(t),\zeta(t-r)\big), 
  \quad \lambda\geq 0,
\end{equation}
where $r>0$ and $a\colon\R\to\R$ is continuous, $T$-periodic and with
average $\avrg{a}=\frac{1}{T}\int_0^Ta(t)dt \ne 0$.  We are interested
in the $T$-periodic solutions of the above equation.  Without loss of
generality we can assume that $T\geq r$ (see, e.g.\
\cite{fuspa:1997}).  In fact, for $n \in \N$, equation \eqref{dde-ro}
and
\begin{equation*}
  \dot \zeta (t) = a(t)\Phi\big(\zeta(t)\big) +\lambda \Xi\Big(t,\zeta(t),\zeta\big(t-(r - n T)\big)\Big), 
  \quad \lambda\geq 0,
\end{equation*}
share the same $T$-periodic solutions (although other solutions may be
radically different). Thus, if necessary, one can replace $r$ with $r
-nT$, where $n \in N$ is chosen such that $0 < r -nT \leq T$.

Let us now establish some further notation.
Given any $p\in N$, denote by $\t p\in\t N$ the constant function $\t
p(t)\equiv p$, $t\in [-T, 0]$. Moreover, for any $V\subseteq N$, and
$W\subseteq\t N$ we define the sets
\begin{equation*}
  \su V:=\big\{\t p\in\t N\colon p\in V\big\},
\end{equation*}
and
\begin{equation*}
  \giu W:=\big\{p\in N\colon \t p\in W \big\}. 
\end{equation*}
Notice also that, for any given $V\subseteq N$, one has $\su
V\subseteq\t V$ and $\giu{(\t V)}=V$.  \smallskip

Proceeding as in \cite[\S \ 3]{fuspa:1997}, we now introduce a
Poincar\'e-type $T$-translation operator on an open subset of $\t
N$. Here, we assume that $\Phi$ is $C^1$.  Let $Q^{\Phi}_T$ be the map
defined, whenever it makes sense for $\phi\in\t N$, by
\[
Q^{\Phi}_T(\phi)(\theta) = \zeta\big(\phi(0), T + \theta\big),\quad
\theta\in [-r, 0],
\]
where $\zeta(p,\cdot)$ denotes the unique maximal solution of the
Cauchy problem
\begin{subequations} \label{ddeg}
  \begin{gather}
    \dot \zeta(t)  = \Phi(\zeta(t)), \label{nopert}\\
    \zeta(0) =p.
  \end{gather}
\end{subequations}
Well known properties of differential equations imply that the domain
$\dom{Q^{\Phi}_T}$ of $Q^{\Phi}_T$ is an open subset of $\t
N$. Moreover, since $T\geq r$, the Ascoli-Arzel\`a Theorem implies
that $Q^{\Phi}_T$ is a locally compact map (see, e.g.\ \cite{Ol69}).
Observe that, there is a simple relation between $\dom{Q^{\Phi}_T}$
and $\dom{P^{\Phi}_T}$, that is
\[
\dom{Q^{\Phi}_T}=\{\varphi\in\t M:\varphi(0)\in\dom{P^\Phi_T}\}.
\]
In particular, $\t{\dom{P^\Phi_T}}\subseteq\dom{Q^{\Phi}_T}$. Notice
also that $P^\Phi_T(p)=Q^\Phi_T(\t p)(0)$ for all
$p\in\dom{P^{\Phi}_T}$.

Similarly, given $a\colon\R\to\R$ as in \eqref{dde-ro}, we define the
map $Q^{a\Phi}_T$ by setting, whenever it makes sense for $\phi\in\t
N$,
\[
Q^{\Phi}_T(\phi)(\theta) = \xi\big(\phi(0), T + \theta\big),\quad
\theta\in [-r, 0],
\]
where $\xi(p, \cdot)$ is the unique maximal solution of the Cauchy
problem
\begin{subequations} \label{a-ddeg}
  \begin{gather}
    \dot \xi(t)  = a(t)\Phi\big(\xi(t)\big),\label{a-nopert} \\
    \xi(0) =p.
  \end{gather}
\end{subequations}
Since $\dom{P^{\Phi}_T}= \dom{P^{a\Phi}_T}$, it follows easily that
$\dom{Q^{\Phi}_T}=\dom{Q^{a\Phi}_T}$.  \smallskip

It is not difficult to prove that the $T$-periodic solutions of
\eqref{nopert} are in a one-to-one correspondence with the fixed
points of $Q^{\Phi}_T$. Similarly, the $T$-periodic solutions to
\eqref{a-nopert} are in a one-to-one correspondence with the fixed
points of $Q^{a\Phi}_T$. Moreover, if $\avrg{a}=1$, Proposition
\ref{SVDAEs:equivalence-T-res} imply that the fixed points of
$P^\Phi_T$ coincide with those of $P^{a\Phi}_T$.  However, even in
this case, $Q^{\Phi}_T$ might be different from $Q^{a\Phi}_T$.  We
wish to obtain a formula for the fixed point index of admissible pairs
$(Q^{a\Phi}_T,W)$, with $W$ open in $\dom{Q^{a\Phi}_T}$. In the case
when $a(t)\equiv 1$, we have the following result (\cite[Theorem
3.2]{FS09}).

\begin{theorem}\label{th2.1}
  Let $\Phi$ be as above and let $W\subseteq\t N$ be open and such
  that $\ind(Q^{\Phi}_T,W)$ is defined.  Then, $\deg(-\Phi,\giu W)$ is
  defined as well and
  \begin{equation}\label{findeg1}
    \ind(Q^{\Phi}_T,W)=\deg(-\Phi,\giu W).
  \end{equation}
\end{theorem}

It is not difficult to see that, for any constant $c$ and any tangent
vector field $v$, admissible on an open $\mathcal{V}\subseteq N$, one
has
\begin{equation}\label{degmoltip}
  \deg(-c\,v,\mathcal{V})=(-\sign c)^{\dim N}\deg(v,\mathcal{V}).
\end{equation}
Hence, when $a(t)\equiv\avrg{a}$, Equation \eqref{findeg1} yields
\[
\ind(Q^{\avrg{a}\Phi}_T,W)=\deg(-\avrg{a}\,\Phi,\giu W)
=(-\sign\avrg{a})^{\dim N}\deg(\Phi,\giu W).
\]

We seek to generalize this formula to the case when $a$ is
nonconstant. The first part of our construction follows that of the
proof of Theorem 3.2 in \cite{FS09}.

\begin{theorem}\label{thgrado}
  Let $a$, $\Phi$ and $T$ be as in \eqref{dde-ro} and let
  $Q^{a\Phi}_T$ be as above. Let also $W\subseteq\t{N}$ be open. If
  the fixed point index $\ind(Q^{a\Phi}_T,W)$ is defined, then so is
  $\deg(\Phi,\giu W)$ and
  \begin{equation} \label{DSVeq:deg-deg}
    \ind(Q^{a\Phi}_T,W)=(-\sign\avrg{a})^{\dim N}\deg(\Phi,\giu W).
  \end{equation}
  In particular, one has that
  \begin{equation} \label{DSVeq:deg-deg-deg}
    \ind(Q^{a\Phi}_T,W)=(-\sign\avrg{a})^{\dim N}\ind(Q^{\Phi}_T, W).
  \end{equation}
\end{theorem}

\begin{proof}
  The assumption that $\ind(Q^{a\Phi}_T,W)$ is defined means that
  $W\subseteq\dom{Q^{a\Phi}_T}$ and that the fixed point set $\fix
  (Q^{a\Phi}_T)\cap W$ is compact. Let us show that $\deg(\Phi,\giu
  W)$ is defined too. We need to prove that $\Phi^{-1}(0)\cap\giu W$
  is compact. If $p\in\Phi^{-1}(0)\cap\giu W$, then the constant
  function $\t p$ is clearly a fixed point of $Q^{a\Phi}_T$. Thus
  $\Phi^{-1}(0)\cap\giu W$ is compact since it can be regarded as a
  closed subset of $\fix (Q^{a\Phi}_T)\cap W$.

  We now use the Commutativity Property of the fixed point index in
  order to obtain a relation between the indices of $P^{a\Phi}_T$ and
  $Q^{a\Phi}_T$. Define the maps $h:\dom{P^{a\Phi}_T}\to\t M$ and
  $k:\t M\to M$ by $h(p)(\theta)=\xi(p,\theta+T)$ and
  $k(\phi)=\phi(0)$, respectively.  Here, $\xi(p, \cdot)$ indicates
  the unique maximal solution of the Cauchy problem
  \eqref{a-ddeg}. One has that
  \begin{subequations}\label{compo}
    \begin{equation}\label{compo1}
      (h\circ k)(\phi)(\theta) =\xi\big(\phi(0),\theta+T\big)=Q^{a\Phi}_T(\phi)(\theta),
      \quad\phi\in\dom{Q^{a\Phi}_T},\quad \theta\in[-r,0],
    \end{equation}
    and
    \begin{equation}\label{compo2}
      (k\circ h)(p) =\xi(p,\theta+T)\vert_{\theta=0}= 
                     \xi(p,T)=P^{a\Phi}_T(p), \quad p\in\dom{P^{a\Phi}_T}.
    \end{equation}
  \end{subequations}
  Define $\gamma=k|_W$. As a consequence of the Commutativity Property
  of the fixed point index,
  $\ind\Big(h\circ\gamma,\gamma^{-1}\big(\dom{P^{a\Phi}_T}\big)\Big)$
  is defined if and only if $\ind\big(\gamma\circ h,h^{-1}(W)\big)$ is
  defined as well and, in this case,
  \begin{equation}\label{indcomm}
    \ind\Big(h\circ\gamma,\gamma^{-1}\big(\dom{P^{a\Phi}_T}\big)\Big)
    =\ind\big(\gamma\circ h,h^{-1}(W)\big).
  \end{equation}
  Moreover, since $W\subseteq\dom{Q^{a\Phi}_T}$, then
  $\gamma^{-1}\big(\dom{P^{a\Phi}_T}\big)=W$. Hence, from
  \eqref{compo}, it follows that
  \begin{subequations}\label{indQP}
    \begin{equation}
      \ind(Q^{a\Phi}_T,W)=\ind\Big(h\circ\gamma,
      \gamma^{-1}\big(\dom{P^{a\Phi}_T}\big)\Big),
    \end{equation}
    and
    \begin{equation}
      \ind\big(P^{a\Phi}_T,h^{-1}(W)\big)=
      \ind\big(\gamma\circ h,h^{-1}(W)\big).
    \end{equation}  
  \end{subequations}
  Recall that, according to remark \ref{SVDAEs:remark-on-ivp},
  $\mathcal{D}(P^{\Phi}_T)=\mathcal{D}(P^{a\Phi}_T)$ so that
  $h^{-1}(W)\subseteq \mathcal{D}(P^{\Phi}_T)$. Then by
  \eqref{indcomm} and \eqref{indQP} we get
  \begin{equation}\label{indQPh}
    \ind(Q^{a\Phi}_T,W)=\ind\big(P^{a\Phi}_T,h^{-1}(W)\big).
  \end{equation}

  Assume now in addition that the average $\avrg{a}$ of $a$ is equal
  to $1$. Proposition \ref{prop:iaphi=iphi} yields
  \begin{equation}\label{vecchiaf}
    \ind\big(P^{a\Phi}_T,h^{-1}(W)\big)=
    \ind\big(P^\Phi,h^{-1}(W)\big)=\deg\big(-\Phi,h^{-1}(W)\big).
  \end{equation}
  By the definition of $h$, one has that $\Phi^{-1}(0)\cap\giu
  W=\Phi^{-1}(0)\cap h^{-1}(W)$. In fact, all the constant solutions
  of \eqref{a-nopert} lie in $\giu W$.  Then, from the Excision
  Property of the degree of a tangent vector field, we get
  \begin{equation}\label{idovvia}
    \deg\big(-\Phi,h^{-1}(W)\big)= \deg(-\Phi,\giu W).
  \end{equation}
  Therefore, we get \eqref{DSVeq:deg-deg} by \eqref{indQPh},
  \eqref{vecchiaf} and \eqref{idovvia}.\smallskip

  Let us now remove the additional assumption on $a$. Let us put
  $a_0(t)=a(t)/\avrg{a}$ for all $t\in\R$
  and $\Phi_a(p)=\avrg{a}\,\Phi(p)$ for all $p\in N$. We rewrite
  equation \eqref{a-ddeg} as follows:
  \[
  \dot \xi(t) = a_0(t)\Phi_a\big(\xi(t)\big),
  \]
  and observe that $Q^{a_0\Phi_a}_T=Q^{a\Phi}_T$. Since the average of
  $a_0$ over $[0,T]$ is equal to $1$ we get, using the first part of
  the proof,
  \begin{equation}\label{genavgdeg}
    \ind(Q^{a\Phi}_T,W)=\deg(-\,\avrg{a}\,\Phi,\giu W).
  \end{equation}
  Since from \eqref{degmoltip} we have
  \[
  \deg(-\,\avrg{a}\,\Phi,\giu W)=(-\sign\avrg{a})^{\dim
    N}\deg(\Phi,\giu W),
  \]
  the assertion follows from \eqref{genavgdeg}.
\end{proof}

\section{Branches of starting pairs to \eqref{DSVeq:0}}

Any pair $(\lambda,\varphi)\in [0,\infty)\X\t N$ is said to be a
\emph{starting pair} for \eqref{DSVeq:0} if the following initial
value problem has a $T$-periodic solution:
\begin{equation}\label{due.uno}
  \left\{
    \begin{array}{ll}
      \dot \zeta(t)=a(t)\Phi\big(\zeta(t)\big)+\lambda 
      \Xi\big(t,\zeta(t),\zeta(t-r)\big), &  t>0,\\
      \zeta(t)=\varphi(t), & t\in [-r,0].
    \end{array}
  \right.
\end{equation}
A pair of the type $(0,\t p)$ with $\Phi(p)=0$ is clearly a starting
pair and will be called a \emph{trivial starting pair}.  The set of
all starting pairs for (\ref{DSVeq:0}) will be denoted by $S$.
For the reminder of this section we assume that $\Phi$ and $\Xi$ are $C^1$, so
that (\ref{due.uno}) admits a unique solution that we denote by
$\xi^\lambda(\varphi,\cdot)$.  By known continuous dependence
properties of delay differential equations the set
$\mathcal{V}\subseteq[0,\infty)\times\t N$ given by
\begin{equation*}
  \mathcal V:=\big\{(\lambda,\varphi)\in
  [0,\infty)\X\t N: \xi^\lambda(\varphi,\cdot)
  \textrm{ is defined on $[0,T]$}\big\}
\end{equation*}
is open. Clearly $\mathcal V$ contains the set $S$ of all starting
pairs for \eqref{DSVeq:0}.  Observe that $S$ is closed in $\mathcal
V$, even if it may be not so in $[0,+\infty)\times\t N$.  Moreover, by
the Ascoli-Arzel\`a Theorem it follows that $S$ is locally compact.

\smallskip It is convenient to introduce the following notation for
the ``slices'' of product spaces. Let $Y$ be a set. Given $X\subseteq
[0,\infty)\X Y$, we put $X_\lambda=\{\varphi\in Y:(\lambda,\varphi)\in
X\}$ for each $\lambda\geq 0$.

The following is our main result concerning starting pairs:

\begin{theorem}\label{T.3.1}
  Assume that $\Phi$, $\Xi$, $S$ are as above and let $a\colon\R\to\R$
  be continuous and $T$-periodic such that its average on a period is
  nonzero. Let $W\subseteq[0,\infty)\times\t N$ be open. If
  $\deg\big(\Phi,\giu{(W_0)}\big)$ is defined and nonzero, then the
  set
  \begin{equation*}
    (S\cap W)\setminus \big\{(0,\t p)\in W\colon \Phi(p)=0\big\}
  \end{equation*}
  of nontrivial starting pairs in $W$, admits a connected subset whose
  closure in $S\cap W$ meets $\big\{(0,\t p)\in W:\Phi(p)=0 \big\}$
  and is not compact.
\end{theorem}

The proof of Theorem \ref{T.3.1} is based on the following global
connectivity result of \cite{FP93} and follows closely that of
Proposition 4.1 of \cite{FS09}. We adapt it here for the sake of
completeness.
\begin{lemma}\label{L.3.1}
  Let $Y$ be a locally compact metric space and let $Z$ be a compact
  subset of $Y$.  Assume that any compact subset of $Y$ containing $Z$
  has nonempty boundary. Then $Y\setminus Z$ contains a connected set
  whose closure (in $Y$) intersects $Z$ and is not compact.
\end{lemma}

\begin{proof}[Proof of Theorem \ref{T.3.1}]
  Consider the open set $U=W\cap \mathcal V$. Since
  $\Phi^{-1}(0)\cap\giu{(U_0)}=\Phi^{-1}(0)\cap\giu{(W_0)}$, and
  $S\cap U=S\cap W$, we need to prove that the set of nontrivial
  starting pairs in $U$ admits a connected subset whose closure in
  $S\cap U$ meets the set $\big\{(0,\t p)\in U:\Phi(p)=0\big\}$ and is
  not compact. We deduce this fact from Lemma \ref{L.3.1} applied to
  the pair
  \begin{equation*}
    (Y,Z)=\Big(S\cap U,\, \big\{(0,\t p)\in U: p\in \Phi^{-1}(0)\big\} \Big).
  \end{equation*}
  In fact, if a connected set is as in Lemma \ref{L.3.1} then its
  closure satisfies all the requirements.

  Since $U$ is open and $S$ is locally compact $S\cap U$ is locally
  compact too. Moreover, the assumption that
  $\deg\big(\Phi,\giu{(W_0)}\big)$ is defined means that the set
  \begin{equation*}
    \big\{p\in\giu{(W_0)}: \Phi(p)=0\big\}=\big\{p\in\giu{(U_0)}:
    \Phi(p)=0\big\}
  \end{equation*}
  is compact. Thus the homeomorphic set $\{(0,\t p)\in U:\Phi(p)=0\}$
  is compact as well. Let us now prove that there exists no compact
  subset of $S\cap U$ containing $Z$ and with empty boundary in $S\cap
  U$.

  Assume by contradiction that such a set, call it $C$, exists. Then
  $C$ is relatively open in $S \cap U$ and $(S\cap U)\setminus C$ is
  closed in $S\cap U$. Hence, the distance $\delta=\dist\big(C,(S\cap
  U)\setminus C\big)$ between the compact set $C$ and $(S\cap
  U)\setminus C$ is nonzero. Define, for each $\lambda\geq 0$, the map 
$\Q_\lambda\colon\mathcal{V}_\lambda\to\t M$ given by
\begin{equation*}
  \Q_\lambda (\varphi)(\theta)= \xi^\lambda(\varphi,\theta+T),\quad 
  \theta\in[-r,0].
\end{equation*}
Notice that $\Q_0$ coincides with the map $Q^{a\Phi}_T$ defined in the
previous section. In fact, if $\zeta(p,\cdot)$ is the unique solution
of the Cauchy problem \eqref{a-ddeg}, then we have
$\xi^0\big(\varphi(0),\cdot\big)=\zeta\big(\varphi(0),\cdot\big)$.
Consider the set
  \begin{equation*}
    A=\left\{(\lambda,\varphi)\in U 
      :\dist\big((\lambda,\varphi),C\big)<\delta/2\right\},
  \end{equation*}
  which, clearly, does not meet $(S\cap U)\setminus C$. The
  compactness of $S\cap U\cap A=C$ imply that for some $\lambda_*>0$
  one has $\big(\{\lambda_*\}\X A_{\lambda_*}\big)\cap S\cap
  U=\emptyset$. So, since the set $S\cap U\cap A$ coincides with
  $\{(\lambda,\varphi)\in A:\Q_\lambda(\varphi)=\varphi\}$, the
  Generalized Homotopy Invariance Property of the fixed point index
  imply that
  \begin{equation*}
    0=\ind\big(\Q_{\lambda_*}, A_{\lambda_*}\big)=\ind\big(\Q_0,A_0\big),
  \end{equation*}
  Thus, as $Q^{a\Phi}_T=\Q_0$, Theorem \ref{thgrado} and the Excision
  Property of the degree yield
  \begin{equation*}
    0=\ind\big(\Q_0, A_0\big)=\ind(Q^{a\Phi}_T,A_0)=
    \left|\deg(\Phi,\giu{(A_0)})\right|
    =\left|\deg(\Phi,\giu{(W_0)})\right|,
  \end{equation*}
  against the assumption.
\end{proof}

\section{Branches of $T$-periodic pairs to
  \eqref{DSVeq:0}}\label{sec:rami}

In this section we focus on the $T$-periodic solutions to
\eqref{DSVeq:0}. In fact, we study the topological structure of the
set of pairs $(\lambda,x)\in [0,\infty)\X C_T(N)$ where $x$ is a
solution of this equation.  Our result extends that of \cite[\S
5]{FS09}

Recall that $\Phi\colon N\to\R^d$, $\Xi\colon\R\X N\X N\to\R^d$ and
$a\colon\R\to\R$ are continuous with $\Phi$ and $\Xi$ tangent to $N$
in the sense specified in the Introduction. We also assume that $a$
and $\Xi$ are $T$-periodic in $t$ and $a$ has nonzero average on a
period.

A pair $(\lambda,\zeta)\in [0,\infty)\X\ctnn$, where $\zeta$ a is
$T$-periodic solution of \eqref{DSVeq:0}, is called a
\emph{$T$-periodic pair}.  Those $T$-periodic pairs that are of the
particular form $(0,\cl{p})$ are said to be \emph{trivial}. Notice
that, since $a$ is not identically zero, $(0,\cl
p)\in[0,\infty)\X\ctnn$ is a trivial $T$-periodic pair if and only if
$\Phi(p)=0$.  We point out that if $\zeta$ is a nonconstant
$T$-periodic solution of the unperturbed equation $\dot
\zeta=a(t)\Phi(\zeta)$, then $(0,\zeta)$ is a nontrivial $T$-periodic
pair.

The following is our main result. Its proof follows closely that of
\cite[Thm.\ 5.1]{FS09} (which, in turn, is inspired to
\cite{FP93}). In fact, the only remarkable difference is related to
the use of Theorem \ref{T.3.1}. For the sake of completeness, however,
we restate the argument here in a slightly more schematic form.

\begin{theorem}\label{tuno}
  Let $a$, $\Phi$ and $\Xi$ as above. Let $\Omega\subseteq
  [0,\infty)\times \ctnn$ be open and such that $\deg(\Phi,\Omega\cap
  N)$ is defined and nonzero. Then $\Omega$ contains a connected set
  of nontrivial $T$-periodic pairs for \eqref{DSVeq:0} whose closure
  in $\Omega$ meets the set $\{(0,\cl p)\in\Omega:\Phi(p)=0\}$ and is
  not compact. In particular, the set of $T$-periodic pairs for
  \eqref{DSVeq:0} contains a connected component that meets $\{(0,\cl
  p)\in\Omega:\Phi(p)=0\}$ and whose intersection with $\Omega$ is not
  compact.
\end{theorem}

The following lemma takes care of a special case.
\begin{lemma}\label{lemmatuno}
  Let $a$, $\Phi$, $\Xi$ and $\Omega$ be as in Theorem
  \ref{tuno}. Assume in addition that $\Phi$ and $\Xi$ are $C^1$. Then
  $\Omega$ contains a connected set of nontrivial $T$-periodic pairs
  for \eqref{DSVeq:0} whose closure in $\Omega$ meets the set
  $\{(0,\cl p)\in\Omega:\Phi(p)=0\}$ and is not compact.
\end{lemma}

\begin{proof}
  Denote by $X$ the set of $T$-periodic pairs of \eqref{DSVeq:0} and
  by $S$ the set of starting pairs of the same equation. Define the
  map $h \colon X\to S$ by
  $h(\lambda,\zeta)=\big(\lambda,\zeta|_{[-r,0]}\big)$ and observe
  that $h$ is continuous and onto. Since $\Phi$ and $\Xi$ are $C^1$,
  then $h$ is also one to one.  Observe that, the trivial solution
  pairs correspond to trivial starting points under this
  homeomorphism.  Moreover, by continuous dependence on data,
  $h^{-1}:S\to X$ is continuous as well. Consider the set
  \begin{equation*}
    S_\Omega=\big\{ (\lambda ,\varphi)\in S:\mbox{the solution of \eqref{DSVeq:0} is contained in }\Omega\big\}
  \end{equation*}
  so that $X\cap\Omega$ and $S_\Omega$ correspond under the
  transformation $h:X\to S$. Thus, $S_\Omega$ is an open subset of
  $S$. Consequently, we can find an open subset $W$ of
  $[0,\infty)\times\t N$ such that $S\cap W=S_\Omega$. This implies
  the following chain of equalities:
  \begin{multline*}
    \big\{ p\in\giu{ (W_0)} \colon \Phi(p)=0\big\}=
    \big\{p\in N:(0,\t p)\in W,\;\Phi(p)=0\big\}=\\
    =\big\{p\in
    N:(0,\cl{p})\in\Omega,\;\Phi(p)=0\big\}=\big\{p\in\Omega\cap
    N:\Phi(p)=0\big\}.
  \end{multline*}
  Hence, by excision, it follows that
  $\deg(\Phi,\giu{(W_0)})=\deg(\Phi,\Omega\cap N)\neq 0$. By appealing
  to Theorem \ref{T.3.1}, we find that there exists a connected set
  \begin{equation*}
    \Sigma\subseteq ( S\cap W)\setminus\big\{(0,\t p)\in W:\Phi(p)=0\big\}
  \end{equation*}
  whose closure in $S\cap W$ meets $\big\{(0,\t p)\in
  W:\Phi(p)=0\big\}$ and is not compact.

  The set $\Gamma = h^{-1}(\Sigma)\subseteq X\cap\Omega$ is a
  connected set of nontrivial $T$-periodic pairs whose closure in
  $X\cap\Omega$ meets $\{(0,\cl p)\in\Omega:\Phi(p)=0\}$ and is not
  compact. Since $X$ is closed in $[0,\infty)\times \ctnn$, the
  closures of $\Gamma$ in $X\cap\Omega$ and in $\Omega$ coincide.
  Therefore $\Gamma$ satisfies the requirements.
\end{proof}

The proof of Theorem \ref{tuno} can be now performed through an
approximation procedure.

\begin{proof}[Proof of Theorem \ref{tuno}]
  As in the last part of the proof of Lemma \ref{lemmatuno}, it is
  enough to show the existence of a connected set $\Gamma$ of
  nontrivial $T$-periodic pairs whose closure in $X\cap\Omega$ meets
  $\{(0,\cl p)\in\Omega:\Phi(p)=0\}$ and is not compact.

  Observe that the closed subset $X$ of $[0,\infty)\X\ctnn$ is locally
  compact because of Ascoli-Arzel\`a Theorem. It is convenient to
  introduce the following subset of $X$:
  \[
  \Upsilon=\big\{(0,\cl{p})\in [0,\infty)\X\ctnn:\Phi(p)=0\big\}.
  \]
  Take $Y=X\cap\Omega$ and $Z=\Upsilon\cap\Omega$ and notice that $Y$
  is locally compact as an open subset of $X$.  Moreover, $Z$ is a
  compact subset of $Y$ (recall that, by assumption,
  $\deg(\Psi,N\cap\Omega)$ is defined).  Since $Y$ is closed in
  $\Omega$, we only have to prove that the pair $(Y,Z)$ satisfies the
  hypothesis of Lemma \ref{L.3.1}. Assume the contrary. Thus, we can
  find a relatively open compact subset $C$ of $Y$ containing $Z$.
  Similarly to the proof of Proposition \ref{T.3.1}, given
  $0<\rho<\dist (C,Y\setminus C)$, we consider the set $A^\rho$ of all
  pairs $(\lambda,\varphi)\in\Omega$ whose distance from $C$ is
  smaller than $\rho$. Thus, $A^\rho\cap Y=C$ and $\partial A^\rho\cap
  Y=\emptyset$. We can also assume that the closure $\cl{A^\rho}$ of
  $A^\rho$ in $[0,\infty)\X\ctnn$ is contained in $\Omega$. Since $C$
  is compact and $[0,\infty)\times N$ is locally compact, we can take
  $A^\rho$ in such a way that the set
  \[
  \left\{ \big(\lambda,x(t),x(t-r)\big) \in [ 0,\infty)\times N\X
    N: (\lambda ,x)\in A^\rho,\; t\in [0,T]\right\}
  \]
  is contained in a compact subset of $[0,\infty)\X N\X N$. This
  implies that $A^\rho$ is bounded with complete closure and
  $A^\rho\cap N$ is a relatively compact subset of $\Omega\cap N$. In
  particular $\Phi$ is nonzero on the boundary of $A^\rho\cap N$
  (relative to $N$). Known approximation results on manifolds yield
  sequences $\{\Phi_i\}_{i\in\N}$ and $\{ \Xi_i\}_{i\in\N}$ of $C^1$ maps uniformly
  approximating $\Phi$ and $\Xi$, respectively, and such that for all $i\in\N$ we
  have
  \begin{list}{---}{\setlength{\itemsep}{3pt}%
      \setlength{\topsep}{3pt}%
      \setlength{\partopsep}{0pt}%
      \setlength{\parsep}{0pt}%
      \setlength{\parskip}{0pt}}
  \item[({\it a})] $\Phi_i(p)\in T_pN$ for all $p\in N$;
  \item[({\it b})] $\Xi_i(t,p,q)\in T_pN$ for all $(t,p,q)\in\R\X N\X
    N$;
  \item[({\it c})] $\Xi_i$ is $T$-periodic in the first variable.
  \end{list}
  Thus, for $i\in\N$ large enough, we get
  \[
  \deg (\Phi_i,A^\rho\cap N)=\deg (\Phi,A^\rho\cap N).
  \]
  Furthermore, by excision,
  \[
  \deg (\Phi,A^\rho\cap N)=\deg (\Phi,\Omega \cap N)\neq 0.
  \]
  Therefore, given $i$ large enough, Lemma \ref{lemmatuno} can be
  applied to the equation
  \begin{equation}\label{due}
    \dot x(t)=a(t)\Phi_i\big(x(t)\big)+\lambda\Xi_i\big(t,x(t),x(t-r)\big). 
  \end{equation}
  Let $X_i$ denote the set of $T$-periodic pairs of \eqref{due} and
  put
  \[
  \Upsilon_i=\big\{(0,\cl{p})\in [0,\infty)\times
  \ctnn:\Phi_i(p)=0\big\}.
  \]
  Lemma \ref{lemmatuno} yields a connected subset $\Gamma_i$ of
  $A^\rho$ whose closure in $A^\rho$ meets $\Upsilon_i\cap A^\rho$ and
  is not compact. Let us denote by $\cl\Gamma_i$ and $\cl{A^\rho}$ the
  closures in $[0,\infty)\times\ctnn$ of $\Gamma_i$ and $A^\rho$,
  respectively.

  We claim that, for $i$ large enough, $\cl\Gamma_i\cap\D
  A^\rho\neq\emptyset$. Thus, $X_i$ being closed, we get
  $\cl{\Gamma}_i\subseteq X_i$ and this implies the existence of a
  $T$-periodic pair $(\lambda_i,x_i)\in\D A^\rho$ of \eqref{due}.

  Notice that proving the claim boils down to showing that
  $\cl\Gamma_i$ is compact. In fact, if the latter assertion is true
  and we assume $\cl\Gamma_i\cap\D A^\rho=\emptyset$, then we get
  $\cl\Gamma_i\subseteq A^\rho$ so that the closure of $\Gamma_i$ in
  $A^\rho$ coincides with the compact set $\cl\Gamma_i$. But this is a
  contradiction.
  The compactness of $\cl\Gamma_i$ for $i$ large enough follows
  from the completeness of $\cl{A^\rho}$ and the fact that, by the
  Ascoli-Arzel\`a Theorem, $\cl\Gamma_i$ is totally bounded, when $i$
  is sufficiently large. Hence the claim is proved.

  Again by Ascoli-Arzel\`a Theorem, we may assume that, as $i\to\infty$, 
$x_i\to x_0$ in $\ctnn$ and $\lambda_i\to\lambda_0$ with $(\lambda_0,x_0)\in\D
  A^\rho$. Passing to the limit in equation \eqref{due}, it is not
  difficult to show that $(\lambda_0,x_0)$ is a $T$-periodic pair for
  \eqref{DSVeq:0} in $\partial A^\rho$. This contradicts the
  assumption $\partial A^\rho\cap Y=\emptyset$ and proves the first
  part of the assertion.

  Let us prove the last part of the thesis. Consider the connected
  component $\Xi$ of $X$ that contains the connected set $\Gamma$ of
  the first part of the assertion.  We shall now show that $\Xi$ has
  the required properties. Clearly, $\Xi$ meets the set
  $\big\{(0,\cl{p})\in\Omega:\Phi(p)=0\big\}$ because the closure
  $\cl{\Gamma}^\Omega$ of $\Gamma$ in $\Omega$ does. Moreover,
  $\Xi\cap\Omega$ cannot be compact, since $\Xi\cap\Omega$, as a
  closed subset of $\Omega$, contains $\cl{\Gamma}^\Omega$, and
  $\cl{\Gamma}^\Omega$ is not compact. This completes the proof.
\end{proof}

\section{Applications and examples}
The purpose of this section is to illustrate the techniques developed
and the results obtained in the foregoing ones. In order to do so, we
will examine two classes of separated variables perturbed differential
equations.  Namely, perturbed decoupled systems and
differential-algebraic equations of a certain form.

\subsection{Weakly coupled equations}
Here, we consider delay periodic perturbations of a particular family
of ordinary differential equations on product manifolds. Namely, if
$N_1\subseteq \R^{n_1}$ and $N_2\subseteq \R^{n_2}$ are boundaryless
smooth manifolds, we consider the following differential equation on
$N=N_1\X N_2$:
\begin{equation} \label{SVDAEs-two-species-with-delay}
  \left\{ \begin{array}{l}
      \dot x_1(t) = a_1(t)\Phi_1\big(x_1(t)\big) + 
      \lambda \Xi_1\big(t, x(t), x(t-r)\big), \\
      \dot x_2(t) = a_2(t)\Phi_2\big(x_2(t)\big) + \lambda
      \Xi_2\big(t, x(t), x(t-r)\big),
    \end{array} \right.  
\end{equation} 
where $\Phi_1\colon N_1\to\R^{n_1}$, $\Phi_2\colon N_2\to\R^{n_2}$,
are (continuous) tangent vector fields, $\Xi_1\colon\R\X N\X
N\to\R^{n_1}$, $\Xi_2\colon\R\X N\X N\to\R^{n_2}$ and the maps
$a_i:\R\to\R$, $i=1,2$, are continuous, $T$-periodic in $t$. We also
assume that, for $i=1,2$, the average of $\avrg{a}_i$ of $a_i$ is
nonzero and that $\Xi_i$ is tangent to $N_i$ in the second variable.

Clearly, when $\lambda=0$, the resulting unperturbed equations are
completely decoupled. In essence, the perturbation provides the (only)
coupling in \eqref{SVDAEs-two-species-with-delay}.

We point out that equations of this form arise naturally from models
for a real system. Indeed, the equations
\eqref{SVDAEs-two-species-with-delay} are inspired to a model
(see, e.g.\ \cite[Chapter 2]{Fa:2001}) describing the interactions of
two species that share the same habitat and feed on the same resource.
Namely, if $x_1(t)$ and $x_2(t)$ are the densities of such a competing
species at time $t$, the model is as follows:
\begin{equation} \label{DSVeq:two-spec} \left\{\begin{array}{l}
      \dot x_1= a_1x_1 + x_1 (a_{12}x_2 - a_{11}x_1),\\
      \dot x_2= a_2x_2 + x_2 (a_{22}x_2 - a_{21}x_1),
    \end{array}\right.
\end{equation}
where $a_i$, $a_{j,k}$, $i,j,k=1,2$, are positive quantities and
$a_1$, $a_2$ represent the ``intrinsic'' growth rates of the two
species, $a_{11}$, $a_{22}$ give the strength of the intraspecific
competition and $a_{12}$, $a_{21}$ the strength of the interspecific
competition.
It may be reasonable (see e.g.\ \cite{kuang}) to assume that the
intrinsic growth rates undergo periodic fluctuations. We can model
this by letting $a_i$, $i=1,2$, be periodic functions. If we also
suppose that the resource on which $x$ and $y$ feed takes time $r$ to
recover, we are let to the following modification of
\eqref{DSVeq:two-spec}:
\begin{equation*} 
  \left\{\begin{array}{l} 
      \dot x_1(t)= a_1(t)x_1(t) + x_1(t) \big(a_{12}x_2(t-r) - a_{11}x_1(t-r)\big),\\
      \dot x_2(t)= a_2(t)x_2(t) + x_2(t) \big(a_{22}x_2(t-r) - a_{21}x_1(t-r)\big).
    \end{array}\right.
\end{equation*}
Let us point out that more general examples in the same direction are
considered in a number of models for the dynamics of animals
populations (see, e.g.\ \cite{brauer-cast, kuang}).

The system \eqref{SVDAEs-two-species-with-delay}, where a parameter
$\lambda\geq 0$ has been added can be thought as a far-reaching
generalization of the above system. However, our little digression
above is only intended as justification of our interest in Equation
\eqref{SVDAEs-two-species-with-delay}. In fact, in this section, we
consider solutions of this equation regardless of any possible
ecological meaning.

Even if Theorem \ref{tuno} cannot be applied directly to
\eqref{SVDAEs-two-species-with-delay}, we can use the same
strategy. We only sketch the argument. As in Remark
\ref{SVDAEs:remark-on-ivp}, assume that, for $i=1,2$, $\Phi_i$ are
$C^1$ so that uniqueness of the solutions of the following initial
value problems on $N=N_1\X N_2=N$ hold:
\begin{subequations}
  \begin{equation}\label{DSVeq:A-more}
    \dot x_1 = \Phi_1(x_1),\quad \dot x_2 = \Phi_2(x_2),\qquad x(0)=\xi_0,
  \end{equation}
  and
  \begin{equation}\label{DSVeq:B-more}
    \dot x_1 = a_1(t)\Phi_1(x_1),\quad  \dot x_2 = a_2(t)\Phi_2(x_2),\qquad  x(0)=\xi_0,
  \end{equation}
\end{subequations}
and let $\xi\colon J\to N$ and $u\colon I\to N$ be the maximal
solutions of \eqref{DSVeq:A-more} and \eqref{DSVeq:B-more},
respectively, with $I$ and $J$ the relative maximal intervals of
existence. Let $t>0$ be such that $\int_0^{l} a_i(s) ds \in I$,
$i=1,2$, for all $l\in [0,t]$, then it follows that
\begin{equation*}
  \xi(t)=\big(\xi_1(t),\xi_2(t)\big)
  =\bigg(u_1\Big(\int_0^t a_1(s) ds\Big),u_2\Big(\int_0^t a_2(s) ds\Big)\bigg),
\end{equation*} 
and hence $t\in J$.  Conversely, by a maximality argument, it can be
shown that $T\in J$ implies $\int_0^ta_i(s)ds\in I$, $i=1,2$.
%
%

As in Section \ref{SVADEs:section-poincare-trasl-op} we construct an
``infinite dimensional'' $T$-translation operator associated to
\ref{SVDAEs-two-species-with-delay} for $\lambda =0$. Namely, we let
$Q^{a_1,a_2}_T$ be the map that to any $\varphi\in\t N$ associates the
map $\theta\mapsto\zeta\big(T+\theta,\varphi(0)\big)$, whenever it
makes sense to do so. Here $\zeta(\cdot,p)$ denotes the unique
solution of \eqref{DSVeq:A-more} with $\xi_0=\varphi(0)$.

Let $\Phi\colon N\to\R^{n_1}\X\R^{n_2}=\R^{n_1+n_2}$ be the tangent
vector field on $N$ given by
$\Phi(p_1,p_2)=\big(\Phi_1(p_1),\Phi_2(p_2)\big)$. The following
result similar to Theorem \ref{thgrado} holds:

\begin{proposition} \label{SVDAEs:poincare-transl-op-(a-b)} Let $a_1$,
  $a_2$, $\Phi_1$, $\Phi_2$, $N_1$, $N_2$, $N$, $T$ and $Q^{a_1,a_2}$ be as above. Take
  $W\subseteq\t N$ open and such that the fixed point index of
  $Q^{a_1,a_2}$ is defined in $W$, then so is $\deg(\Phi, \giu W)$ and
  \begin{equation} \label{DSVeq:deg-deg-b}
    \ind(Q^{a_1,a_2},W)=(\sign\avrg{a}_1)^{\dim
      N_1}(\sign\avrg{a}_2)^{\dim N_2}\deg(\Phi,\giu W).
  \end{equation}
\end{proposition}
\begin{proof}[Sketch of the proof.]
  The assertion can be proved by following closely the argument of
  Theorem \ref{thgrado} and taking into account the following
  well-known and easily verified fact of degree theory:

  For any given pair of constants $c_1,c_2\in\R\setminus\{0\}$ and tangent
  vector fields $v_1\colon N_1\to\R^{n_1}$ and
  $v_2\colon N_2\to\R^{n_2}$, admissible on an open
  $\mathcal{V}\subseteq N$, one has
  \begin{equation*}
    \deg(v_c,\mathcal{V})=(-\sign c_1)^{\dim N_1}(-\sign c_2)^{\dim N_2}
    \deg(v,\mathcal{V}).
  \end{equation*}
  where $v, v_c\colon N\to\R^{n_1+n_2}$ are the tangent vector fields on
  $N$ given by $v(p_1,p_2)=\big(v_1(p_1),v_2(p_2)\big)$ and 
  $v_c(p_1,p_2)=\big(c_1v_1(p_1),c_2v_2(p_2)\big)$ for
  all $(p_1,p_2)\in N$.
\end{proof}

Any $(\lambda,\zeta)\in[0,\infty)\X\ctnn$, with $\zeta$ solution of
\eqref{SVDAEs-two-species-with-delay}, is a \emph{$T$-periodic
  pair}. Such a pair is \emph{trivial} if $\lambda=0$ and $\zeta$ is
constant.  An argument that follows very closely the one of Theorem
\ref{tuno} yields the following result:

\begin{proposition} \label{DSVeq:corollary-wcoupled} For $i=1,2$, let
  $\Phi_i\colon N_i\to\R^{n_i}$, be (continuous) tangent vector
  fields, and let $\Xi_i\colon\R\X N\X N\to\R^{n_i}$ be tangent to
  $N_i$ in the second variable; assume that the $\Xi_i$'s as well as
  the maps $a_i:\R\to\R$, are continuous, $T$-periodic in $t$. Suppose
  also that, for $i=1,2$, the average of $\avrg{a}_i$ of $a_i$ is
  nonzero.  Let $\Omega\subseteq [0,\infty)\times C_T(N)$ be open, and
  assume that $\deg\big(\Phi,\Omega\cap\t N\big)$ is defined and
  nonzero. Then $\Omega$ contains a connected set of nontrivial
  $T$-periodic pairs of \eqref{SVDAEs-two-species-with-delay} whose
  closure in $\Omega$ meets the set $\big\{(0,\cl p)\in\Omega:
  \Phi(p)=(0)\big\}$ and is not compact.
\end{proposition}

\begin{example}
  Let $T = 2\pi$ and consider the following system of equations in $\R^3$
  \begin{equation*}\label{first-example}
    \left\{ \begin{array}{l}
        \dot x_2  = \big(2 +\sin (t)\big) (x_2+x_3),\\
        \dot x_1 - \dot x_3  = |\cos (t)| x_1  - \big(2 +\sin (t)\big)x_3, \\
        -\dot x_2 + \dot x_3 =  -\big(2 +\sin (t)\big) x_2, \\
      \end{array}\right.
  \end{equation*}
  that we write more compactly as follows:
  \begin{equation}\label{matrix-form} 
E \dot x =  A(t) x,
  \end{equation}
  where, for $t\in\R$,
  \begin{equation*}
    E=\begin{pmatrix}
      0 & 1 & 0 \\
      1 & 0 & -1 \\
      0 & -1 & 1 \\
    \end{pmatrix}
    \quad \text{and}\quad A(t)=
    \begin{pmatrix}
      0          &  2+ \sin (t)           & 2+ \sin (t) \\
      |\cos (t)| & 0                      &  -2 -\sin (t) \\
      0  & -2 -\sin (t) & 0\\
    \end{pmatrix}.
  \end{equation*}
  Let us now consider the following $2\pi$-periodic perturbation of \eqref{matrix-form}:
  \begin{equation} \label{perturbed} 
E \dot x(t) =  A(t)x(t) + \lambda \mathcal{H}\big(t, x(t),x(t-r)\big), \quad \lambda\geq 0,
  \end{equation}
  where $r>0$ is a given time lag and $\mathcal{H}\colon\R\X\R^3\X\R^3
  \to\R^3$ is a continuous map which is $2\pi$-periodic in its first
  variable.
  Multiplying \eqref{perturbed} on the left by $E^{-1}$ and setting, for all $(t,p,q)\in\R\X\R^3\X\R^3$,
\begin{equation*}
H(t, p,q)=E^{-1}\mathcal{H}(t, p, q),\quad\text{and}\quad B(t)=E^{-1}A(t)
  \end{equation*}
we see that \eqref{perturbed} becomes
 \begin{equation}\label{DSVeq:ex-wcoupled}
 \dot x(t)=B(t)x(t)+\lambda H\big( t,x(t),x(t-r)\big).
 \end{equation}
Clearly,
\[
 B(t)=\begin{pmatrix}
       |\cos(t)| &  0        &  0       \\
       0         & 2+\sin(t) &  0       \\
       0          &  0       & 2+\sin(t)
      \end{pmatrix}
\begin{pmatrix}
 1 & 0 & 0 \\
 0 & 1 & 1 \\
 0 & 0 & 1 
\end{pmatrix}.
\]
Hence, identifying $\R^3$ with the product space $\R\X\R^2$, we see that
\eqref{DSVeq:ex-wcoupled} falls into the family of weakly coupled systems 
\eqref{SVDAEs-two-species-with-delay}  with $N_1=\R$ and $N_2=\R^2$.
Let $\Phi\colon\R\X\R^2\to\R^3$ be given by $\Phi(p;q_1,q_2):=(p,q_1+q_2,q_1)$, 
and take $\Omega=[0,\infty)\X C_T(\R^3)$.  Since, $\Phi$ is admissible for the degree 
in $\Omega\cap\R^3=\R^3$ and $\deg\big(\Phi,\R^3\big) = 1$, by Proposition 
\ref{DSVeq:corollary-wcoupled}, there exists a connected set of nontrivial $T$-periodic 
pairs of \eqref{DSVeq:ex-wcoupled} whose closure meets the set $\Phi^{-1}(0)=\{(0;0,0)\}$ 
and is not compact. Since solutions of \eqref{DSVeq:ex-wcoupled} are solutions of 
\eqref{perturbed} and vice versa, this statement concerns, in fact, the $T$-periodic 
solutions of \eqref{perturbed}.
\end{example}

\subsection{Periodic perturbations of separated variables Differential-Algebraic Equations}
Here, as in the Introduction, $U\subseteq\R^k\times\R^s$ is open and
connected and $g\colon U\to\R^s$ is $C^\infty$ with the property that
$\partial_2g(x,y)$ is nonsingular for any $(x,y)\in U$. In this way,
$M:=g^{-1}(0)$ is a $C^\infty$ submanifold of $\R^k\times\R^s$. We
also require that $f$ and $h$, as in Equation \eqref{DSVeq:0a}, are
continuous and that $a$ and $h$ are $T$-periodic in $t$ with the
average of $a$ different from zero.

In what follows, we say that $\big(\lambda,
(x,y)\big)\in[0,\infty)\X C_T(U)$ is a $T$-periodic pair of
\eqref{DSVeq:0a}, if $(x,y)$ is a $T$-periodic solution of
\eqref{DSVeq:0a} corresponding to $\lambda$.  According to the
convention introduced in
\eqref{DSVeq:graph}-\eqref{DSVeq:convention-constan-func}, any $(p,
q)\in U$ is identified with the element $(\cl{p},\cl{q})$ of $C_T(U)$
that is constantly equal to $(p, q)$. A $T$-periodic pair of the form
$\big(0,(\cl p,\cl q)\big)$ will be called \emph{trivial}. This
subsection is devoted to the study of the set of $T$-periodic pairs of
equation \eqref{DSVeq:0a}.

Thanks to our assumption on $g$ it is possible to associate tangent
vector fields on $M$ to the functions $f$ and $h$ in
\eqref{DSVeq:0a}. Consider first maps $\Psi\colon U\to\rkrs$ and
$\Upsilon\colon\R\X U\X U\to\rkrs$ as follows:
 \begin{subequations}
\begin{align*}
 &\Psi (p_1, q_1)= \big( f(p_1, q_1), -[\D_2 g(p_1,q_1)]^{-1} \D_1g (p_1, q_1) f(p_1, q_1) \big), \\[2mm]
&\begin{aligned} 
    \Upsilon\big(t, &(p_1, q_1), (p_2, q_2)\big)= \\
    &\Big( h\big(t, (p_1, q_1), (p_2, q_2)\big), -[\D_2 g(p_1, q_1
    )]^{-1}\D_1 g (p_1, q_1) h\big(t, (p_1, q_1), (p_2,q_2)\big)
    \Big),
  \end{aligned}\!\!\!
\end{align*}
 \end{subequations}
and then define
\begin{equation}\label{campiv}
  \Phi=\Psi|_M\quad\text{ and }\quad \Xi=\Upsilon|_{\R\X M\X M}.
\end{equation}
Since $T_{(p,q)}M$ coincides with the kernel of the differential of
$g$ at any $(p,q)\in M$, it can be easily seen that $\Phi(p,q)\in
T_{(p,q)}M$ and that $\Xi\big(t,(p_1,q_1),(p_2,q_2)\big)\in T_{(p_1,
  q_1)}M$, for all $\big(t,(p_1, q_1),(q_2,p_2)\big)\in\R\X M\X
M$. Therefore, the following is a delay differential equation on $M$:
\begin{equation} \label{DSVeq:eq-su-M} \dot \zeta(t) =
  a(t)\Phi\big(\zeta(t)\big) +
  \lambda\Xi\big(t,\zeta(t),\zeta(t-r)\big),\quad \lambda\geq 0.
\end{equation} 
Using Lemma \ref{lemmequiv} it is not difficult to show that
\eqref{DSVeq:eq-su-M} is equivalent to \eqref{DSVeq:0a}, in the sense
that $\zeta=(x,y)$ is a solution of \eqref{DSVeq:eq-su-M}, on an
interval $I\subseteq\R$, if and only if so is $(x,y)$ for
\eqref{DSVeq:0a}.
Thus, we can combine the results of Theorems \ref{DSVeq:teo1} and
\ref{tuno}.  The map $F\colon U\to\rkrs$ introduced in \eqref{DSVeq:F}
becomes, in our case,
\begin{equation*}
  F(x,y)= \big(f(x,y),g(x,y)\big).
\end{equation*}
So, with the notation recalled above, the set of trivial $T$-periodic
pairs can be written as $\big\{(0, (\cl{p}, \cl{q})) \in [0,\infty) \X
C_T (U) : F(p,q) = (0, 0)\big\}$. Also, as in Section \ref{sec:rami},
given $ \Omega\subseteq [0,\infty) \X C_T (U)$, we denote by
$\Omega\cap U$ the subset of $U$ whose points, regarded as constant
functions, lie in $\Omega$. Namely, $\Omega \cap U =\big\{(p, q) \in U
: (0, (\cl{p}, \cl{q})) \in\Omega\big\}$.

We finally state and prove the following consequence of Theorems
\ref{DSVeq:teo1} and \ref{tuno}, which is inspired to \cite[Th.\
5.1]{CaSp}.

\begin{theorem} \label{DSVeq:main-thm} Let $U \subseteq \R^k \X \R^s$
  be open and connected.  Let $g : U \to \R^s$, $f : U \to \R^k$, $a :
  \R\to \R$ and $h : \R \X U \to \R^k$ be as above.  Let also $F(p, q)
  = \big(f(p, q), g(p, q)\big)$. Given $\Omega \subseteq [0,\infty) \X
  C_T (U)$ open, assume that $\deg(F,\Omega\cap U)$ is well-defined
  and nonzero.  Then, there exists a connected set of nontrivial
  solution pairs of \eqref{DSVeq:0a} whose closure in $\Omega$ is
  noncompact and meets the set $\big\{(0, (\overline{p},
  \overline{q})) \in \Omega : F(p, q) = (0, 0)\big\}$ of the trivial
  $T$-periodic pairs of \eqref{DSVeq:0a}.
\end{theorem}
\begin{proof}
  Let $\Phi$ and $\Xi$ be the tangent vector fields defined in
  \eqref{campiv}. Let also $\mathcal{O}$ be the open subset of
  $[0,\infty)\X C_T(M)$ given by
  \begin{equation*}
    \mathcal{O}=\Omega\cap \Big([0,\infty)\X C_T(M)\Big).
  \end{equation*} 
  For any $Y\subseteq M$, by $\mathcal{O}\cap Y$ we mean the set of
  all those points of $Y$ that, regarded as constant functions, lie in
  $\mathcal{O}$. Using this convention, one has that $\Omega\cap
  Y=\mathcal{O}\cap Y$ and, in particular, $\Omega \cap M
  =\mathcal{O}\cap M$. Thus, Theorem \ref{DSVeq:teo1} implies that
  \begin{equation*}
    |\deg(\Phi, \mathcal{O} \cap M)| = |\deg(\Phi,\Omega\cap M)|
    = |\deg\big(F,\Omega \cap U\big)|\neq 0.
  \end{equation*} 
  Theorem \ref{tuno} yields a connected set $\Lambda$ of nontrivial
  $T$-periodic pairs of \eqref{DSVeq:eq-su-M} whose closure in
  $\mathcal{O}$ is not compact and meets the set
  \[
  \big\{\big(0,(\cl p,\cl q)\big)\in\mathcal{O}:\Phi(p,q)=(0,0)\big\}=
  \big\{(0,\cl p,\cl q)\in\Omega:F(p,q)=(0,0)\big\}.
  \]
  The equivalence of \eqref{DSVeq:eq-su-M} with \eqref{DSVeq:0a} imply
  that each $(\lambda,(x,y))\in\Lambda$ is a nontrivial $T$-periodic
  pair of \eqref{DSVeq:0a} as well. Since $M$ is closed in $U$, any
  relatively closed subset of $\mathcal{O}$ is relatively closed in
  $\Omega$ too and vice versa. Thus, the closure of $\Lambda$ in
  $\mathcal{O}$ coincides with the closure of $\Lambda$ in $\Omega$,
  and hence $\Lambda$ fulfills the assertion.
\end{proof}

\end{document}